\documentclass[review]{siamart1116}

\usepackage{latexsym}
\usepackage{amsmath}
\usepackage{amssymb}
\usepackage{amsfonts}
\usepackage[latin1]{inputenc}
\usepackage{mathrsfs}
\usepackage{amsfonts}
\usepackage{graphicx}
\usepackage{colortbl,dcolumn}
\usepackage{amsmath}
\usepackage{amsfonts}
\usepackage{graphicx}
\usepackage{algorithmic}
\ifpdf
  \DeclareGraphicsExtensions{.eps,.pdf,.png,.jpg}
\else
  \DeclareGraphicsExtensions{.eps}
\fi

\newtheorem{tm}{Theorem}[section]
\newtheorem{rk}{Remark}[section]

\newtheorem{prop}{Proposition}[section]
\newtheorem{lm}{Lemma}[section]
\newtheorem{cor}{Corollary}[section]

\newcommand{\E}{\mathbb E}

\newcommand{\bi}{\mathbf i}

\newcommand{\<}{\langle}
\renewcommand{\>}{\rangle}

\newcommand{\cui}[1]{{\color{red} [cui: #1]}}
\newcommand{\zhou}[1]{{\color{brown} [zhou: #1]}}

\ifpdf
  \DeclareGraphicsExtensions{.eps,.pdf,.png,.jpg}
\else
  \DeclareGraphicsExtensions{.eps}
\fi

\newcommand{\TheTitle}{Optimal control for stochastic nonlinear Schr\"odinger equation on graph} 
\newcommand{\TheAuthors}{Jianbo, Cui and Shu, Liu and Haomin, Zhou}

\headers{}{\TheAuthors}

\title{{\TheTitle}\thanks{}}

\title{{\TheTitle}\thanks{The research is partially supported by Georgia Tech Mathematics Application Portal (GT-MAP) and by research grants NSF  DMS-1830225, and ONR N00014-21-1-2891. The research of the first author is partially supported by the Hong Kong Research Grant Council ECS grant 25302822, start-up funds (P0039016, P0041274) from Hong Kong Polytechnic University and the CAS AMSS-PolyU Joint Laboratory of Applied Mathematics. 
}}

\author{Jianbo Cui 
\thanks{Department of Applied Mathematics, The Hong Kong Polytechnic University, Hung Hom, Kowloon, Hong Kong 
(\email{jianbo.cui@polyu.edu.hk}, corresponding author)
}
\and 
Shu Liu
\thanks{School of Mathematics, Georgia Institute of Technology, Atlanta, GA 30332, USA
(\email{sliu459@gatech.edu})}
\and 
Haomin Zhou
\thanks{School of Mathematics, Georgia Institute of Technology, Atlanta, GA 30332, USA
(\email{hmzhou@gatech.edu}))}
}

\usepackage{amsopn}

\ifpdf
\hypersetup{
  pdftitle={\TheTitle},
  pdfauthor={\TheAuthors}
}
\fi

\usepackage{ulem}
\usepackage{subfigure}
\usepackage{graphicx}
\usepackage{enumerate}
\usepackage{amsmath,bm}
\usepackage{latexsym}
\usepackage{float}
\allowdisplaybreaks[4]

\begin{document}

\maketitle

\begin{abstract}
We study the optimal control formulation for stochastic nonlinear Schr\"odinger equation (SNLSE) on a finite graph.
By viewing the SNLSE as a stochastic Wasserstein Hamiltonian flow on density manifold, we show the global existence of a unique strong solution for SNLSE with a linear drift control or a linear diffusion control on graph. Furthermore, we provide the gradient formula, the existence of the optimal control and a description on the optimal condition via the forward and backward stochastic differential equations.
\end{abstract}



\section{Introduction}

The nonlinear Schr\"odinger equation (NLSE) given in the form of 
\begin{align*}
\hbar \bi \frac {\partial }{\partial t} \Psi(t,x)=-\frac {\hbar^2}2 \Delta \Psi(t,x)+\Psi(t,x)\mathbb V(x)+\Psi(t,x)f(|\Psi(t,x)|^2)
\end{align*}
has wide applications in quantum mechanics, quantum optics, nuclear physics, transport and diffusion phenomena, and Bose-Einstein condensations (see, e.g., \cite{Sch31,SS99,Caz03}). The unknown $\Psi(t,x)$ represents a complex wave function for $x\in\mathbb R^d, \hbar >0$ is the Planck constant, $\mathbb V(\cdot)$ and $f(\cdot)$ are real-valued functions, referred as linear and nonlinear interaction potentials respectively.
Considering the randomness in the propagation of nonlinear dispersive waves, the stochastic nonlinear Schr\"odinger equation (SNLSE)
\begin{align}\label{SNLSE}
\hbar \bi d \Psi(t,x)&=-\frac {\hbar^2}2 \Delta \Psi(t,x)dt+\Psi(t,x)\mathbb V(x)dt+\Psi(t,x)f(|\Psi(t,x)|^2)dt\\\nonumber
&-\bi u(t,x) \mu(x)dt+ u(t,x)  dW(t,x),
\end{align}
 has been introduced and studied in recent years (see, e.g., \cite{BD99,BM14,BHW19,CHS19a,CS21}). 
 Here $W$ is a colored Wiener process (see, e.g. \cite{Dap92}) 
 defined by
 \begin{align*}
 W(t,x)=\sum_{j=1}^{N}\mu_je_j(x)\beta_j(t),\; t\ge 0, \;x\in \mathbb R^d, 
 \end{align*} 
 and 
 \begin{align*}
 \mu(x)=\frac 12 \sum_{j=1}^N |\mu_j|^2|e_j(x)|^2, x\in \mathbb R^d
 \end{align*}
 with $N\in \mathbb N \cup {\infty},$ $\mu_j\in \mathbb C,$ $e_j$ real-valued function and $\beta_j$ independent Brownian motion on a complete filtrated probability space  $(\Omega,\mathcal F, \{\mathcal F_t\}_{t\ge 0},\mathbb P).$
Another physical significance of SNLSE is related to the theory of measurements continuous in time in quantum mechanics and open quantum system (see, e.g., \cite{barchielli2009quantum,MR1348380}).

In this paper, we focus on two types of SNLSEs on a finite graph $G=(V,E,w)$ and their related stochastic control problems. Here $V$ is the vertex set, $E$ is the edge set and $w_{jl}$ is the weight of the edge $(j,l)\in E$ satisfying $\omega_{lj}=\omega_{jl}>0$ if there is an edge between nodes $j$ and $l$, and $\omega_{jl}=0$ otherwise. Throughout this paper, we assume that $G$ is an undirected, connected graph with no self loops or multiple edges. The first type is  the nonlinear Schr\"odinger equation with random perturbation, 
\begin{align}\label{snls-per}
\bi d u_j=(-\frac 12 (\Delta_G u)_j+u_j\mathbb V_j+u_jf_j(|u|^2))dt+\sigma_ju_j\circ dW_t.
\end{align}
Here  $\Delta_G$ is a nonlinear discretization of Laplacian operator on $G$ introduced in \cite{CLZ19} (see \eqref{non-lap} for its formula), $f_j:\mathbb R \rightarrow \mathbb R$ is a continuous real-valued  function, $\mathbb V_j$ is a given linear potential on the node $j$,  $\sigma_j\in  \mathbb R$ represents the diffusion coefficient, and $\{W_t\}_{t\ge 0}$ is one dimensional Brownian motion on $(\Omega,\mathcal F, \{\mathcal F_t\}_{t\ge 0},\mathbb P)$.
The stochastic differential $\circ dW_t$ is understood in the Stratonovich sense. A typical example of the nonlinear function $f_j$ is that $f_j(|u|^2)=\sum_{l=1}^N \mathbb W_{jl} |u_l|^2$ with an interactive potential $\mathbb W_{jl}=\mathbb W_{lj}$ for any $(j,l)\in E.$ We would like to remark that  Eq. \eqref{snls-per} can be viewed as a spatial discretization of Eq. \eqref{SNLSE} when $G$ is a lattice obtained by discretizing a continuous domain (see e.g. \cite{CHL16b}).
Another type is the nonlinear Schr\"odinger equation with white noise dispersion 
\begin{align}\label{snls-whi}
\bi d u_j=-\frac 12 (\Delta_G u)_j\circ dW_t+(u_j\mathbb V_j+u_jf_j(|u|^2))dt.
\end{align}
When $G$ is a lattice,
\eqref{snls-whi} becomes a spatial discretization of NLSE with white noise dispersion  \cite{MR2652190}, which describes the propagation of a signal in an optical fiber with dispersion management. 

Our current investigation is motivated by several reasons. 
Firstly, the Schr\"odinger equation on graph and its control problem have their own interest and applications \cite{MR2898889,MR3103255,Noja12,eilbeck2003discrete}.
Secondly, in contrast to the extensive literature on the optimal control problem and exact controllability of Schr\"odinger equations on continuous domain in both the deterministic and stochastic cases (see, e.g.,\cite{MR3427687,MR2299629,MR3082471,MR3462390,MR3424692,MR3813983}), far fewer results are known when the problem settings are on graphs. One of the main difficulties lies on the weak regularization effect of free Schr\"odinger group and the nonlinear Laplacian operator on graph \cite{CLZ19}. Another one arises from the compact embedding theorem in probability space. Last but not least, both NLSE and SNLSE on a lattice graph can be viewed as a semi-discretization of NLSE and SNLSE on a continuous domain respectively \cite{CHL16b}, hence can be used as numerical schemes to compute (stochastic) optimal control problems involving SNLSEs in practice. However, many challenging questions remain open, such as the preservation of mass, energy, and symplectic structures, and the convergence analysis of semi-discretization of SNLSEs (see, e.g., \cite{CLZ20a} for more discussions). 

Inspired by the optimal control of quantum mechanical system \cite{MR949169,MR2355901}, we shall study an optimal control problem associated with \eqref{snls-per} or \eqref{snls-whi}. 
Formally, we can view their solution  $u=u(j,t,\omega), t\ge 0,\omega\in \Omega,$ as the quantum state or the nonlinear wave at time $t$. The stochastic perturbation may represents an inaccurate measurement via the quantum observation or a dispersion management in optical fiber.  
The optimal control problem considered here is to find an input potential $\mathbb V$ (or a diffusion coefficient $\mathbb \sigma$) such that the state $u(T)$ is as close as possible to a target state $f^1(T)$ and a trajectory $Z^1$, and achieves the  minimum cost (see sections \ref{sec-control} and \ref{sec-condition} for more details). A different viewpoint for this problem is to recover the quantum mechanical potential $\mathbb V$ or a diffusion coefficient $\sigma$ from the observation of the quantum state or the nonlinear wave $u(T)$ at the end of $[0,T]$. Despite many fruitful results on the continuous optimal control problems for NLSE and SNLSE \cite{MR1070712, MR2377449,MR3427687,MR2299629,MR3082471,MR3462390,MR3424692,MR3813983}, a few exist for the problem defined on a graph. To the best of our knowledge, no result has been reported for stochastic control systems with \eqref{snls-per} or \eqref{snls-whi}. 

In this work we study both linear drift and diffusion control. 
Our approach is based on two key ideas. One is used by Nelson in his derivation for NLSE \cite{Nelson19661079}. The other is viewing SNLSE as a stochastic Wasserstein Hamiltonian flow \cite{CLZ21s}.
By using the complex expression $u=\sqrt{\rho}e^{\bi S}$, we obtain the equivalent Madelung systems of SNLSE on graph (see, e.g., \cite{CLZ19,CS21}). Then by exploiting the properties of Madelung systems, we obtain the existence and uniqueness of the strong solution of \eqref{snls-per} or \eqref{snls-whi} when the control $\mathbb V$ or $\sigma$ is admissible. When the graph is taken as a lattice, we prove that the SNLSE on graph with the nonlinear Laplacian operator preserves the stochastic dispersion relationship, while any linear discretization does not. 
Furthermore, for a quadratic (or convex) cost functional, we provide the gradient formula and prove the existence of the optimal control by carefully studying the probability of tail event of \eqref{snls-per} or \eqref{snls-whi}. When $\sigma$ is a constant potential on every node,
we derive the adjoint equation of \eqref{snls-per} or \eqref{snls-whi} which gives a forward-backward stochastic differential equation and characterizes the necessary optimal condition for the optimal control problem on graph.

 Our paper is organized as follows. In section 2, we explain why we consider the nonlinear Laplacian for the stochastic Schr\"odinger
 equation on graph.
  In section 3, we present some useful properties of the stochastic Schr\"odinger
 equation on graph.
 In section 4, we prove the existence and uniqueness result for \eqref{snls-per} or \eqref{snls-whi} with admissible control variables and prove the existence result of the optimal control.
In section 5, we derive the gradient formula and present 
the necessary optimal condition by deriving a forward-backward stochastic differential equation.

\section{Why nonlinear Laplacian for stochastic Schr\"odinger equation on graph?}
 To explain the reason, 
we consider the stochastic linear Schr\"odinger equation
\begin{align}\label{con-linear-snls1}
\bi d u=-\frac 12 \Delta u dt+\sigma u\circ dW_t.
\end{align}
and the white noise dispersion linear Schr\"odinger equation
\begin{align}\label{con-linear-snls2}
\bi d u=-\frac 12 \Delta u \circ dW_t.
\end{align}
One can directly verify that these equations possess the stochastic dispersion relationship  by It\^o's formula.
\begin{lm}\label{sto-dispersion}
Let $\sigma\in \mathbb R$.
Equation \eqref{con-linear-snls1} (or \eqref{con-linear-snls2}) admits infinitely many plane wave solutions given in the form of $u(x,t)=Ae^{\bi(\mathbb K\cdot x-\mu t-\sigma W(t))}$ (or $Ae^{\bi(\mathbb K\cdot x-\mu  W(t))}$) with arbitrary $A\in \mathbb R^+$, any wave number $\mathbb K\in \mathbb R^d$ and frequency $\mu$ satisfying $\mu=\frac 12 {|\mathbb K|^2}$.

\end{lm}

From the above result, we see that the stochastic dispersion relationship $\mu=\frac 12 {|\mathbb K|^2}$ coincides with the classical dispersion relationship, and the argument of the plane wave contains all the information of the Wiener process. 
However, such a simple property may become problematic in discrete settings. To illustrate where the trouble is, let us consider a lattice $G$ obtained by discretizing $\mathbb R^d$ or $\mathbb T^d$. Any linear discretizations of \eqref{con-linear-snls1} and \eqref{con-linear-snls2} can be stated
\begin{align}\label{dis-rela1}
  \bi d u_j=-\frac 12 \sum_{l\in N(j)} C_{lj} u_l dt+\sigma u_j\circ dW_t  
\end{align}
and 
\begin{align}\label{dis-rela2}
    \bi d u_j=-\frac 12 \sum_{l\in N(j)} C_{lj} u_l \circ dW_t,
\end{align}
respectively. Here $\{C_{lj}\}_{(l,j)\in E}$ are chosen to approximate the Laplacian operator in \eqref{con-linear-snls1} and \eqref{con-linear-snls2}. 
For simplicity, we assume that every node has the same number of adjacent nodes, and that the weight on each edge is uniformly given by $\Delta x$. We denote the coordinate of the node $j$ by $x_j=j\Delta x.$  
Regardless of how $\{C_{lj}\}_{(l,j)\in E}$ are selected, there are at most a finite discrete stochastic plane waves which satisfy the stochastic dispersion relationship.

\begin{tm}
For any linear discretization of \eqref{con-linear-snls1} and \eqref{con-linear-snls2}, there exist at most a finite number of pairs $(\mu,\mathbb K)$ with $\mu=\frac 12 |\mathbb K|^2$ so that the discrete stochastic plane waves, i.e., $u_j=Ae^{\bi(\mathbb K\cdot x_j-\mu t-\sigma W(t))}$ for \eqref{dis-rela1} (or $Ae^{\bi(\mathbb K\cdot x_j-\mu W(t))}$ for \eqref{dis-rela2}), are the solutions. 
\end{tm}

\begin{proof}
Consider the discrete stochastic plane waves $u_j(t)=Ae^{\bi(\mathbb K \cdot x_j-\mu t -\sigma W(t))}$
for \eqref{con-linear-snls1} and $u_j(t)=Ae^{\bi(\mathbb K \cdot x_j-\mu W(t))}$
for \eqref{con-linear-snls2}.
Substituting them into \eqref{dis-rela1}  and \eqref{dis-rela2}, 
we get 
\begin{align*}
    &\mu Ae^{\bi(\mathbb K \cdot x_j-\mu t -\sigma W(t))}dt 
    = \frac 12 \sum_{l\in N(j)} C_{lj} Ae^{\bi(\mathbb K \cdot x_l-\mu t -\sigma W(t))}dt,
\end{align*}
and 
\begin{align*}
    \mu Ae^{\bi(\mathbb K \cdot x_j-\mu W(t))}\circ dW(t) = \frac 12 \sum_{l\in N(j)} C_{lj} Ae^{\bi(\mathbb K \cdot x_l-\mu W(t))}\circ dW(t),
\end{align*}
respectively. 
If $\mu=\frac 12 {|\mathbb K|^2},$ we obtain 
\begin{align*}
    \mu=\frac {|\mathbb K|^2}2= \frac 12 \sum_{l\in N(j)} C_{lj} e^{\bi (\mathbb K \cdot (x_l-x_j))}.
\end{align*}
Since 
 $\frac {|\mathbb K|^2}2$ is quadratic in $\mathbb K$ while the trigonometric polynomial on the right hand side is periodic and bounded in $K$, they intersect only in a bounded ball of the complex domain $|\mathbb K|\le C_N$. Besides, it can be seen that the imaginary part of $\frac 12 \sum_{l\in N(j)} C_{lj} \sin(\mathbb K \cdot (x_l-x_j))=0$ has at most finite zero point. Thus, we complete the proof.
 



\end{proof}

To numerically preserve the stochastic dispersion relationship for any pair of $(\mu,\mathbb K)$ with $\mu=\frac 12 |\mathbb K|^2$, we decide to use the nonlinear Laplacian operator $\Delta_G$ constructed by using the Madelung transformation as shown in \cite{CLZ19,CLZ20a}.

\section{Stochastic nonlinear Schr\"odinger equation on graph}
Consider a graph $G=(V,E,\omega)$,
let us denote the set of discrete probabilities on the graph by
$$\mathcal P(G)=\{(\rho)_{j=1}^N\ :\, \sum_{j=1}^N\rho_j =1, \rho_j\ge 0,\; \text{for} \; j\in V\},$$ and $\mathcal P_o(G)$ as its interior (i.e., all  $\rho_j> 0$, for $j\in V$).
$\mathbb V_j$ is a linear potential on each node $j$, and $\mathbb W_{jl}=\mathbb W_{lj}$ is an 
interactive potential between nodes $j$ and $l$.
We denote $N(i)=\{j\in V: (i,j) \in E\}$ the adjacency set of the node $a_i$ and 
$\theta_{ij}(\rho)$ a density dependent weight on the edge $(i,j)\in E$. More precisely, $\theta$ is defined by $\theta_{ij}(\rho)=\Theta(\rho_i,\rho_j),$ where $\Theta$ is a continuous differentiable function on $(0,1)^2$ satisfying 
$\Theta(x,y)=\Theta(y,x)$, $\Theta(x,y)\ge 0$, and $\min(x,y) \le \Theta(x,y)\le \max(x,y)$
for any $x,y\in (0,1)$. For example, we may take $\theta(\rho)$ as the averaged probability weight in \cite{CLZ19}, i.e., $\Theta(x,y)=\frac 12 {(x+y)}$, or the logarithmic probability  weight in \cite{CLZ20a}, i.e., $\Theta(x,y)=\int_0^1 x^{1-t}y^tdt$, or the harmonic probability  weight in \cite{Mas11}, i.e., $\Theta(x,y)=\frac {2}{1/x + 1/y}$.    

In this section, we present the stochastic nonlinear Schr\"odinger equations on graph via the viewpoint of stochastic variational principle proposed in \cite{CLZ21s}.
Define the total linear potential function $\mathcal V$, interaction potential function $\mathcal W$, and the entropy function $L$ by
$$\mathcal V(\rho)=\sum_{i=1}^N\mathbb V_i\rho_i,\,\,
\mathcal W(\rho)=\frac 12\sum_{i,j=1}^N\mathbb W_{ij}\rho_i\rho_j, \;\; L(\rho)=  \sum_{i=1}^N (\log(\rho_i)\rho_i-\rho_i).$$ 
$I(\rho)$ is 
the discrete Fisher information on graph, i.e.,   
\begin{equation}\label{DiscFisher}
I(\rho)=\frac 12\sum_{i=1}^N\sum_{j\in N(i)}\widetilde \omega_{ij}|\log(\rho_i)-\log(\rho_j)|^2\widetilde \theta_{ij}(\rho),
\end{equation}
where $(\widetilde \omega,\widetilde \theta)$ is another pair of weight and
density dependent weight on the edges $G$. We remark that $(\widetilde \omega,\widetilde \theta)$ may be selected the same as or differently from $(\omega, \theta)$.
Throughout this paper, we take $\theta$ as the averaged probability weight,  $\widetilde \theta$ as the logarithmic probability weight, 
and $\omega_{ij}=\widetilde \omega_{ij}$
for simplicity.

As given in \cite{CLZ2022}, the stochastic variational principe on graph is defined as
\begin{align}\label{gen-var-pri}
      \mathcal I(\rho^0,\rho^T)=\inf\{\mathcal S(\rho_t,\Phi_t)| (-\Delta_{\rho_t})^{\dagger}\Phi_t \in \mathcal T_{\rho_t} \mathcal P_{o}(\mathcal M),\rho(0)=\rho^0, \rho(T)=\rho^T\},
 \end{align}
whose action functional is expressed in the dual coordinates, 
\begin{align*}
\mathcal S(\rho_t,\Phi_t)&=-\<\rho(0),\Phi(0)\>+\<\rho(T),\Phi(T)\>-\int_0^T \<\partial_t \Phi(t),\rho_t\>-\mathcal H_0(\rho_t, \Phi_t) dt\\  &-\int_0^T \mathcal H_1(\rho_t,\Phi_t)\circ d W_t.
\end{align*}
Here $(-\Delta_{\rho})^\dagger$ is the pseudo inverse of $div_G^{\theta}(\rho \nabla_G(\cdot))$ defined by
$$\Big(div_G^{\theta}(\rho\nabla_G(\cdot))\Big)_i:=\sum_{j\in N(i)}\theta_{ij}(\rho)\omega_{ij}(S_j-S_i)$$
for any potential function $S=\{S_i\}_{i\in V}.$
The vector field $\nabla_G S$ induced by $S$ is defined by 
$\nabla_G(S):=\Big(\sqrt{\omega_{ij}}(S_{i}-S_j)\Big)_{ij\in E}.$
With the above notation, one can also introduce the inner product for the vector fields on graph defined by 
$$\<u,v\>_{\theta(\rho)}:=\frac 12\sum_{ij\in E} u_{ij}v_{ij}\theta_{ij}(\rho)\omega_{ij},$$
for any two vector fields (skew-symmetric matrices) $u,v$. 
The kinetic energy is defined by $K(S,\rho)=\frac 12\<\nabla_G S, \nabla_G S\>_{\theta(\rho)}.$
Here $\rho^0,\rho^T$ are $\mathcal F_0$ and $\mathcal F_T$ measurable functions, 
the dominated energy $\mathcal H_0$ and perturbed energy $\mathcal H_1$ are given by 
\begin{align*}
&\mathcal H_0(\rho,S)=K(S,\rho)+ F(\rho)-\kappa L(\rho),\;\\
&\mathcal H_1(\rho,S)=\eta_1 K(S,\rho)+\eta_2 I(\rho)+\eta_3 \Sigma (\rho)+\eta_4\mathcal W(\rho)+\eta_5 L(\rho)
\end{align*}
with  $\kappa\in \mathbb R$, $\Sigma$ defined by $\Sigma(\rho)=\sum_{j=1}^N\sigma_j \rho_j$ for some $\sigma_j\in \mathbb R$, and $ F(\rho):=\frac 18 I(\rho)+\mathcal V(\rho)+\mathcal W(\rho)$.
In particular, when $\eta_1=0$, \eqref{gen-var-pri} recovers the classical variational problem with random potential in Lagrangian formalism.

By finding the critical point of the stochastic variational principle \eqref{gen-var-pri}, we achieve the following discrete stochastic Wasserstein Hamiltonian flow on the density manifold,
\begin{equation}\label{sdhs}\begin{split}
&d\rho=\frac {\partial}{\partial S}\mathcal H_0(\rho,S)+\frac {\partial}{\partial S}\mathcal H_1(\rho,S)\circ dW_t,\\
&d S=-\frac {\partial}{\partial \rho}\mathcal H_0(\rho,S)-\frac {\partial}{\partial \rho}\mathcal H_1(\rho,S)\circ dW_t,
\end{split}\end{equation}

Selecting different deterministic energy $\mathcal H_0$ and perturbed energy $\mathcal H_1$ results in various forms of stochastic nonlinear Schr\"odinger equations on graph.
When $\mathcal H_0(\rho,S)$ $=K(S,\rho)+\mathcal F(\rho)-\kappa L(\rho)$, $\mathcal H_1(\rho,S)=\Sigma(\rho),$ the Wasserstein Hamiltonian flow becomes 
\begin{equation}\label{snls1}\begin{split}
& d\rho_i+\sum_{j\in N(i)}\omega_{ij}(S_j-S_i)\theta_{ij}(\rho)=0,\\
& {d S_i}+\frac 12\sum_{j\in N(i)}\omega_{ij}(S_i-S_j)^2 \frac {\partial \theta_{ij}(\rho)}{\partial \rho_i}dt+\frac 18 \frac {\partial I(\rho)}{\partial \rho_i}dt+\mathbb V_idt\\
&+\sum_{j=1}^N\mathbb W_{ij}\rho_jdt-\kappa \log (\rho_i)dt+\sigma_i dW_t=0.
\end{split}\end{equation}
Its complex formulation $u(t)=\sqrt{\rho(t)}e^{\bi S(t)}$ gives the stochastic nonlinear Schr\"odinger on graph,
\begin{align}\label{snls-per-com}
\bi d u_j=(-\frac 12 (\Delta_G u)_j+u_j\mathbb V_j+u_j\sum_{l=1}^N\mathbb W_{jl} |u_l|^2-u_j\kappa \log(|u_j|^2))dt+\sigma_ju_j\circ dW_t.
\end{align}
Here the nonlinear Laplacian on the graph is defined by
\begin{align}\label{non-lap}
(\Delta_G u)_j&=-u_j\Big(\frac 1{|u_j|^2}\Big[\sum_{l\in N(j)} \omega_{jl}(\Im(\log(u_j))-\Im(\log(u_l))\theta_{jl}) \\\nonumber
&+\sum_{l\in N(j)}\widetilde \omega_{jl} \widetilde \theta_{jl}(\Re(\log(u_j))-\Re(\log(u_l)) \Big] \\\nonumber
&+\sum_{l\in N(j)} \omega_{jl} \frac {\partial \theta_{jl}}{\partial \rho_j}|\Im(\log(u_j)-\log(u_l))|^2 \\\nonumber
&+\sum_{l\in N(j)}\widetilde \omega_{jl} \frac {\partial \widetilde \theta_{jl}}{\partial \rho_j}|\Re(\log(u_j)-\log(u_l))|^2
\Big),
\end{align}
where $\Re$ and $\Im$ are real and imaginary parts of a complex number. This is precisely the nonlinear graph Laplacian introduced in \cite{CLZ20a}.

When $\mathcal H_0=\mathcal V(\rho)+\mathcal W(\rho)$, $\mathcal H_1=K(\rho,S)+\frac 18 I(\rho)$,  the Wasserstein Hamiltonian flow  becomes 
 \begin{align}\label{snls2}
& d\rho_i=\sum_{j\in N(i)}\omega_{ij}(S_i-S_j)\theta_{ij}(\rho) \circ dW_t;\\\nonumber
& d S_i+(\frac 12\sum_{j\in N(i)}\omega_{ij}(S_i-S_j)^2 \frac {\partial \theta_{ij}}{\partial \rho_i}+\frac 18 \frac {\partial}{\partial \rho_i} I(\rho))\circ dW_t+(\mathbb V_i
+\sum_{j=1}^N\mathbb W_{ij}\rho_j)dt=0,
\end{align} 
whose complex formulation $u(t)=\sqrt{\rho(t)}e^{\bi S(t)}$ satisfies  the nonlinear Schr\"odinger equations with white noise dispersion on graph,
\begin{align}\label{snls-whi-com}
\bi d u_j=-\frac 12 (\Delta_G u)_j\circ dW_t+(u_j\mathbb V_j+u_j\sum_{l=1}^N\mathbb W_{jl} |u_l|^2)dt.
\end{align}
Both \eqref{snls-per-com} and \eqref{snls-whi-com} can be viewed as spatial discretization of \eqref{snls-per} and \eqref{snls-whi} respectively when $G$ is a lattice graph.

Recall that in \cite{CLZ19,CLZ20a}, the global solution in deterministic case ($\eta_1=\cdots=\eta_4=\eta_5=0$, $\kappa=0$) is obtained by using the energy conservation law if $\mathcal F(\rho)$ contains the Fisher information $\beta I(\rho), \beta>0$.
In the stochastic case, the existence of global solution has been studied in \cite{CLZ2022} by using the Poisson bracket $\{\cdot, \cdot\}$. 
In particular, when $\{\mathcal H_0,\mathcal H_1\}=0$, for example $\mathcal H_0$ is a multiple of $\mathcal H_1$, then $\mathcal H_0$ is an invariant of the stochastic Wasserstein Hamiltonian flow. 
Here we summarize some fundamental properties shared by the stochastic nonlinear Schr\"odinger equations on graph.

\begin{prop}\label{prop-snls}
Let  $T>0,$ $u(0)$ be $\mathcal F_0$-measurable with any finite moment and $u_j(0)\neq 0$ for all $j\in V$. 
Then \eqref{snls-per-com} (or \eqref{snls-whi-com}) has a unique strong solution  $u(t)$ on $[0,T].$ Moreover, $u(t)$ satisfies the following properties 
\begin{itemize}
\item[(i)] It conserves the total mass
\begin{align*}
\sum_{j=1}^N |u_j(t)|^2=1, \; \text{a.s.}\;;
\end{align*}
\item[(ii)] The total energy satisfies 
\begin{align*}
\E \Big[\sup_{t\in [0,T]} \mathcal E^p(u(t))\Big]\le C(\mathcal E(u(0)), T, p),
\end{align*}
where $\mathcal E$ is defined by a combination of the discrete kinetic energy $\mathcal E_{kin},$ linear potential $\mathcal E_{lin}$, 
interaction potential $\mathcal E_{int}$ and entropy $\mathcal E_{ent},$ i.e.
\begin{align*}
\mathcal E(u)= \mathcal E_{kin}(u)+\mathcal E_{lin}(u)+\mathcal E_{int}(u)+\mathcal E_{ent}(u).
\end{align*}
Here we have 
\begin{align*}
&\mathcal E_{kin}(u)= \frac 14 \sum_{(j,l) \in E}\{|\Re(\log u_j-\log(u_l))|^2\omega_{jl}\theta_{jl}(|u|^2)\\
&+|\Im(\log u_j-\log(u_l))|^2\widetilde \omega_{jl}\theta_{jl}(|u|^2)\},\\
&\mathcal E_{lin}(u)=\sum_{j=1}^N \mathbb V_j|u_j|^2, \; \mathcal E_{int}(u)=\frac 12\sum_{j,l=1}^N\mathbb W_{jl} |u_j|^2|u_l|^2,\\
 &\mathcal E_{ent}(u)=-\kappa \sum_{j=1}^N (\log(|u_j|^2)|u_j|^2-|u_j|^2).
\end{align*}
\item[(iii)] It is time transverse invariant when $\mathbb V$ is independent of time: if $u^{\alpha}(t)$ is the solution of \eqref{snls-per-com} (or \eqref{snls-whi-com}), where $\mathbb V^{\alpha}=(\mathbb V_j+\alpha)_{j=1}^N$ with $\alpha$ being a constant $\mathcal F_0$-measurable random variable, then 
\begin{align*}
u^{\alpha}(t)=u(t)e^{\bi \alpha t},
\end{align*} 
is also a solution.
\item[(iv)] It is time reversible when $\mathbb V$ is independent of time in the following sense: for \eqref{snls-per-com} (or \eqref{snls-whi-com}) with $\widetilde W(t)= W(t), t\ge 0$ and $\widetilde W(t)=-W(-t), t<0$, then
\begin{align*}
u(t)=\bar u(-t).
\end{align*}
\end{itemize}
\end{prop}

\begin{proof}
We can show the existence and uniqueness of $u$ by its complex representation for \eqref{snls-whi-com} and \eqref{snls-per-com}. Thanks to the complex formulation, we know that there always exists $(\rho(0),S(0))$ such that $u(0)=\sqrt{\rho(0)}e^{\bi S(0)}$ with  $|u_i(0)|^2=\rho_i(0)$ such that $\rho_i>0$ for some $i\in V$. The potential $S(0)$ in representation  $(\rho(0),S(0))$ is unique 
up to a shift with $2\pi.$ Let us fix and choose a potential $S(0)$.
Thus to prove the global existence of a unique solution $u$, it suffices to prove that the equivalent systems \eqref{snls1} (or \eqref{snls2}) have a unique global solution. 
To this end, we can use the arguments in \cite[Section 4]{CLZ2022}  and obtain the global existence of the solution. 
The steps to check properties (i)-(iv) are similar to those to prove [Proposition 2.1]\cite{CLZ21s}.

\end{proof}

Following the proof of  \cite[Theorem 4.1]{CLZ2022}, one can also obtain the lower bounds for the density trajectories as stated in the next corollary.
\begin{cor}\label{cor-low}
Let the conditions of Proposition \ref{prop-snls} hold. For Eq.  \eqref{snls1}, there exists a positive random variable which is a lower bound of the density trajectory. 
For Eq. \eqref{snls2}, there exists a positive constant which is a lower bound of the density trajectory. 
\end{cor}

To end this section, we demonstrate that the nonlinear discretization of \eqref{con-linear-snls1} and \eqref{con-linear-snls2} can preserve exactly the stochastic dispersion relationship.
Consider the graph version of \eqref{con-linear-snls1},
\begin{align}\label{linear-snls1}
\bi d u_j= -\frac 12 (\Delta_G u)_j dt+\sigma u_j\circ dW_t.
\end{align}
and that of \eqref{con-linear-snls2},
\begin{align}\label{linear-snls2}
\bi d u_j=-\frac 12 (\Delta_G u)_j\circ dW_t.
\end{align}

\begin{prop} 
Given a lattice graph $G$ with $|x_j-x_l|=\Delta x$ for $l\in N(j)$, $\omega_{ij}=(\frac {\partial \theta_{ij}}{\partial \rho_i}\mathcal N \delta x^2)^{-1}$ where $\mathcal N$ is total number of nodes in $N(j)$ and $\theta_{ij}$ is the symmetric  probability weight. 
The nonlinear discretizations of  \eqref{linear-snls1} and \eqref{linear-snls2} preserve the stochastic dispersion relationship. 
\end{prop}

\begin{proof}
The discrete stochastic plane waves read
$u_j(t)=Ae^{\bi(\mathbb K \cdot x_j-\mu t -\sigma W(t))}$
for \eqref{con-linear-snls1} and $u_j(t)=Ae^{\bi(\mathbb K \cdot x_j-\mu W(t))}$
for \eqref{con-linear-snls2} with $\mu=\frac 12|\mathbb K|^2$.
By the Madelung transformation $u_j=\sqrt{\rho_j}e^{\bi S_j(t)},$ $\rho_j=A$ is constant. As a consequence, the partial derivative of Fisher information $\frac {\partial I(\rho)}{\partial \rho_i}=0.$
On the other hand, since $S_i = \mathbb K \cdot x_i - \mu t - \sigma W(t) $, one can verify that $\frac 12\sum_{j\in N(i)}\omega_{ij}(S_i-S_j)^2 \frac {\partial \theta_{ij}(\rho)}{\partial \rho_i}=\frac{1}{2}|\mathbb{K}|^2=\mu$.
This implies that \begin{align*}
    {d S_i}+\frac 12\sum_{j\in N(i)}\omega_{ij}(S_i-S_j)^2 \frac {\partial \theta_{ij}(\rho)}{\partial \rho_i}dt+\frac 18 \frac {\partial I(\rho)}{\partial \rho_i}dt + \sigma dW_t=0
\end{align*} 
is satisfied.
Thus \eqref{con-linear-snls1} preserves all the stochastic dispersion relationship. 

Similar calculations can show that  \eqref{con-linear-snls2}  satisfy 
\begin{align*}
    {d S_i}+\frac 12\sum_{j\in N(i)}\omega_{ij}(S_i-S_j)^2 \frac {\partial \theta_{ij}(\rho)}{\partial \rho_i}\circ dW_t=0,
\end{align*}
which implies that \eqref{con-linear-snls2} preserves all the stochastic dispersion relationship.


\end{proof}

\section{Stochastic control problem on density manifold of finite graph}
\label{sec-control}

In this section, we propose two stochastic optimal control formulations corresponding to SNLSEs \eqref{snls-per} and \eqref{snls-whi} on graph respectively.

\subsection{Stochastic control problem with linear potential control}
We first assume that the linear potential term $\{\mathbb V_j\}_{j\in N}$ is a control variable depending on $t$. From the the proof of \eqref{prop-snls}, this will not affect the well-posedness of \eqref{snls-per}  and \eqref{snls-whi}.
For convenience, we denote the corresponding solution by $u^{\mathbb V}_j$ in the complex function representation and $(\rho^{\mathbb V}_j, S^{\mathbb V}_j)$ on Wasserstein manifold. 
The admissible control set $\mathcal U$ is defined by 
\begin{align*}
\mathcal U:=&\Big\{\mathbb V: \Omega\times [0,T] \to \mathbb R^N \;\big|\;\mathbb V(t) \;\;\text {is}\;\; \mathcal F_t\text{-adapted}, \mathbb V_j \in L^2([0,T]),\\
&\text{there exists}\; \alpha>0, \; \text{such that} \;  |\mathbb V_j| \le \alpha\; \text{a.s. for}\; j\in V \Big\}
\end{align*}
with $\gamma,\beta \ge0.$
Our first optimal control problem is to minimize the cost functional 
\begin{align}\label{ocp}
J(\mathbb V)&:=\gamma \E\Big[ \sum_{i=1}^N |u^{\mathbb V}_j(T)-f^1_j|^2 \Big]
+\beta \E\Big[\int_0^T \sum_{i=1}^N |\mathbb V_j(t)-Z_j(t)|^2dt\Big], 
\end{align}
subject to the constraint given by either \eqref{snls1} or \eqref{snls2} with given $(\rho(0),S(0)).$
Here $f^1$ is $\mathcal F_T$-adapted satisfying $\|f^1\|_{L^2(\Omega;\mathbb C^N)}<\infty$,  and $Z\in \mathcal U$.
The above optimal control problem may be viewed as the graph 
version of the stochastic control problem in \cite{BRZ16,MR2299629,MR3462390,MR3424692}.

 The following lemma (see, e.g., \cite[Chapter 3]{book1984}) is very useful to show the existence and uniqueness of the optimal control.

\begin{lm}\label{reg}
Let $\mathcal B$ be a uniformly convex Banach space and $\mathcal S$ a bounded closed subset of $\mathcal B$. Furthermore, let $F: \mathcal S \to \overline{\mathbb R}$ be a lower semi-continuous functional which is bounded from below and $p\ge 1$. Then there exists a dense subset $\mathcal D \subset \mathcal B$ such that for each $x\in \mathcal D$, the functional $F(s)+\|s-x\|_\mathcal B^p$ attains its minimum over $\mathcal S,$ which implies that there exists an $s(x)\in \mathcal S$ such that 
\begin{align*}
F(s(x))+\|s(x)-x\|_{\mathcal B}^p=\inf_{s\in \mathcal S}\{F(s)+\|s-x\|_{\mathcal B}^p\}.
\end{align*}
In particular, if $p>1$, then $s(x)$ is unique. Besides, each minimizing sequence converges strongly and the function $x \mapsto s(x)$ is continuous in $\mathcal D.$  
\end{lm}

In our case, we take $\mathcal B:=L^2(\Omega\times[0,T];\mathbb C^N)$ which is uniformly convex, and choose $\mathcal S$ as the admission control set. The functional $F=\gamma \E\Big[ \sum_{i=1}^N |u^{\mathbb V}_j(T)-f^1_j|^2 \Big]$ is bounded from below and $p=2.$
According to Lemma \ref{reg}, if we can verify the lower semi-continuity of $F$, then there exists a dense subset $\mathcal D$ of $\mathcal B$ such that 
for each $Z\in \mathcal D$ the functional $J(\mathbb V)=F(\mathbb V)+\beta\|\mathbb V-Z\|^2_{\mathcal B}$ attains its unique minimum over $\mathcal U.$ In other word, there exists a unique $\mathbb V^*\in \mathcal U$ such that 
\begin{align*}
J(\mathbb V^*)=F(\mathbb V^*)+\beta \|\mathbb V^*-Z\|_{\mathcal B}^2
=\inf_{\mathbb V} J(\mathbb V).
\end{align*}
To prove the lower semi-continuity of $u^{\mathbb V}$ with respect to $\mathbb V$,  we show a strong convergence result first.

\begin{prop}\label{converge}
Let $u(0)$ be $\mathcal F_0$-adapted with any finite moment satisfying $u_j(0)\neq 0,j\le N.$ 
Let the sequence $\{\mathbb V^n\}_{n\ge1}\subset \mathcal U$ be convergent to $\mathbb V$ and $u^{\mathbb V^n}$ be the corresponding solution of the stochastic nonlinear Schr\"odinger equation \eqref{snls1} (or \eqref{snls2}) with respect to the control $\mathbb V^n$ and the initial value $u^{\mathbb V^n}(0)=u(0)$.
Then the sequence $(u^{\mathbb V^n})\in  L^2(\Omega; \mathcal C([0,T];\mathbb C^N)), n\ge 1$, converges strongly to the solution of stochastic nonlinear Schr\"odinger equation \eqref{snls1} (or \eqref{snls2}) with respect to the control $\mathbb V\in \mathcal U.$
\end{prop} 

\begin{proof}
In this proof, we only show the details when the constraint is \eqref{snls1}. A similar argument can lead to the strong convergence result for the case of \eqref{snls2}.
By Proposition \ref{prop-snls}, the It\^o formula, and the Burkholder's inequality,  we have the following {\it a priori} estimates, 
\begin{align}\label{mass}
&\sum_{i=1}^N|u_i^{\mathbb V_n}(t)|^2=\sum_{i=1}^N|u_i(0)|^2=1, \; \text{a.s.}\\\label{energy}
&\E\Big[\sup_{t\in [0,T]}\Big( \<\nabla_G S^{\mathbb V_n}(t), \nabla_G S^{\mathbb V_n}(t)\>_{\theta(\rho^{\mathbb V_n}(t))}+\frac 18 I(\rho^{\mathbb V_n}(t))\Big)^p\Big]\le C(u(0),T,\alpha,p), \; p\ge 1. 
\end{align}
To show the strong convergence of $u^{\mathbb V_n},$ we introduce a stopping time $\tau_c$ defined by 
\begin{align*}
\tau_c^n:=\inf\{t\in [0,T]: \|S^{\mathbb V_n}\|_{\mathcal C([0,t];\mathbb R^N)}\ge c\}\wedge \inf\{t\in [0,T]: \min_{i=1}^N \min_{s\in [0,t]}\rho_i^{\mathbb V_n}(s)\le \frac 1{c}\}.
\end{align*}
By Corollary \ref{cor-low}, we have that $\lim_{c\to \infty} \tau_c =T, a.s.$
Introduce the truncated sample subspace $\Omega_c^n$ defined by 
$$\Omega_c^n=\left\{\sup_{t\in [0,T]}\|S^{\mathbb V_n}\|_{\mathcal C([0,t];\mathbb R^N)}\le c, \min_{i=1}^N \min_{s\in [0,T]}\rho_i^{\mathbb V_n}(s)\ge \frac 1{c}\right\}.$$
Similarly, we denote $\Omega_c$ as the truncated sample subspace with respect to $u^{\mathbb V}.$
Our goal is to show the error estimate in $\Omega_c^n\cap \Omega_c$ and $\Omega/\{\Omega_c^n\cap \Omega_c\}.$
First, we prove the convergence in $\Omega/\{\Omega_c^n\cap \Omega_c\}.$ Due to the mass conservation law \eqref{mass} of the stochastic nonlinear Schr\"odinger equation, by applying the Chebyshev's inequality, we get 
\begin{align*}
&\|1_{\Omega/\{\Omega_c^n\cap \Omega_c\}}(u^{\mathbb V^n}-u^{\mathbb V})\|_{\mathcal B}^2\\
&\le \int_0^T \E \Big[1_{\Omega/\{\Omega_c^n\cap \Omega_c\}} (|u^{\mathbb V^n}(s)|^2+|u^{\mathbb V}(s)|^2)\Big]ds\\
&\le CT \Big[\mathbb P(\sup_{s\in [0,T]}|S^{\mathbb V^n}|\ge c)+\mathbb P(\sup_{s\in [0,T]}|S^{\mathbb V}|\ge c)\\
& +\mathbb P(\min_{i=1}^N \min_{s\in [0,T]}\rho_i^{\mathbb V^n}(s)\le \frac 1{c}) +\mathbb P(\min_{i=1}^N \min_{s\in [0,T]}\rho_i^{\mathbb V}(s)\le \frac 1{c})\Big].
\end{align*}
It suffices to prove all the above probabilities converges to $0$ as $c\to\infty$.
Indeed, since $G$ is connected, by applying the lower bound estimate in \cite[Section 3]{CLZ20a}, there exists a positive random variable $C(\omega)$ such that 
\begin{align}\label{low-est}
\inf_{ t \ge 0}\min_{i\le N}\rho_i^{\mathbb V^n}(t)\ge c_2\exp(-c_1 C(\omega)),
\end{align}
Here $c_2,c_1>0$ are constants  depending on the structure of $G$, and $C(\omega)$ is the positive random variable in Corollary
\ref{cor-low}. More precisely,
the positive random variable $C(\omega)$ is bounded by the upper bound of $\mathbb V^n$ and $\mathbb V$ plus 
\begin{align}\label{rand}
\sup_{t\in [0,T]}\Big( \<\nabla_G S^{\mathbb V^n}(t), \nabla_G S^{\mathbb V^n}(t)\>_{\theta(\rho^{\mathbb V^n}(t))}+\frac 18 I(\rho^{\mathbb V^n}(t))\Big),
\end{align}
which possess any finite moment by \eqref{energy}.
Thus, by \eqref{low-est}, Chebyshev's inequality and the monotonicity of the logarithmic function, we get 
\begin{align}\label{tail-est}
&\mathbb P( \min_{s\in [0,T]}\min_{i=1}^N \rho_i^{\mathbb V^n}(s)\le \frac 1{c}) \\\nonumber
& \le \mathbb P(c_2\exp(-c_1 C(\omega)) \le \frac 1{c}) \\\nonumber  
&=\mathbb P(C(\omega)\ge \frac 1{c_1}(\log(c)+\log(c_2)))\\\nonumber 
&\le \frac {c_1^p\E\Big[C(\omega)^p \Big]}{(\log(c)-\log(c_2))^p},\; p\ge1.
\end{align}
 When $c\to\infty,$ by the dominated convergence theorem, we have that 
\begin{align*}
&\lim_{c\to \infty}\Big[\mathbb P(\min_{i=1}^N \min_{s\in [0,T]}\rho_i^{\mathbb V^n}(s)\le \frac 1{c}) +\mathbb P(\min_{i=1}^N \min_{s\in [0,T]}\rho_i^{\mathbb V}(s)\le \frac 1{c})\Big]=0.
\end{align*}

For the tail estimate of $S^{\mathbb V^n},$ we make use of the differential equation of $S^{\mathbb V^n}$ and get that 
\begin{align*}
|S_i^{\mathbb V^n}(t)|&\le |S_i^{\mathbb V^n}(0)|+\int_0^T\sum_{j\in N(i)} \frac 14 |S_i-S_j|^2\omega_{ij}+|\frac {\partial }{\partial \rho_i} I(\rho)|ds\\
&+\int_0^T |\mathbb V^n_i|+ \sum_{j=1}^N |\mathbb W_{ij}|\rho_jds+\sup_{t\in [0,T]}|\int_0^t \sigma_i dW(s)|.
\end{align*}
The Burkholder's inequality yields that $\E \Big[\sup\limits_{t\in [0,T]}|\int_0^t \sigma_i dW(s)|^p\Big]\le C(p,\sigma).$ Notice that \eqref{energy} and \eqref{low-est} implies that 
\begin{align*}
\max_{ij\in E}|S_i-S_j|^2&\le \frac  2{\min_{ij\in E} \omega_{ij}(\rho_i+\rho_j)}C(\omega)\le  \frac 1{\min_{ij\in E} \omega_{ij}c_2} \exp(c_1 C(\omega)) C(\omega),\; \text{a.s.}\\
\max_{i} |\frac {\partial }{\partial \rho_i} I(\rho)|
&\le \max_{ij} \omega_{ij} \max_{i} [\frac 2{\rho_i}+2|\log(\rho_i)|]\\
&\le \max_{ij} \omega_{ij} 2(\frac 1{c_2}\exp(c_1C(\omega))+|\log(c_2)|+c_1C(\omega))<\infty, \; \text{a.s.}
\end{align*} 

Combining with the fact that $|\mathbb V^n_i(t)|\le \alpha$, we conclude that for $c$ large enough,  
\begin{align*}
&\mathbb P(\sup_{s\in [0,T]}|S^{\mathbb V^n}|\ge c)\\
&\le \mathbb P(  \min_{ij\in E} \frac 1{\omega_{ij}c_2} \exp(c_1 C(\omega)) C(\omega) \ge \frac c {4T})\\
&+\mathbb P(\max_{ij\in E} \omega_{ij} 2(\frac 1{c_2}\exp(c_1C(\omega))+|\log(c_2)|+c_1C(\omega))\ge \frac c{4T})\\
&+\mathbb P( \sup_{i\le N} \sup_{t\in [0,T]}|\int_0^t \sigma_i dW(s)|\ge \frac c{4T})\\
&+\mathbb P(\sup_{i\le N}|S_i^{\mathbb V^n}(0)|+\max_{ij\in E} \mathbb W_{ij} + \alpha T \ge \frac c{4T}).
\end{align*}
Using the moment estimate of $C(\omega)$ and Chebyshev's inequality, we obtain that 
\begin{align*}
\lim_{c\to \infty}\mathbb P(\sup_{s\in [0,T]}|S^{\mathbb V^n}|\ge c)=0.
\end{align*}
Similarly, we can get $\lim_{c\to \infty}\mathbb P(\sup_{s\in [0,T]}|S^{\mathbb V}|\ge c)=0.$ 

On $\Omega_c^n\cap \Omega_c$, we use the stopping time technique to show the strong convergence. By the definition of $\tau_c^n$ and $\tau_c$ we can see that $\tau_c^n= T$ on $\Omega_c^n$ and $\tau_c= T$ on $\Omega_c.$  
According to the complex form of $u^{\mathbb V^n}=\sqrt{\rho^{\mathbb V^n}}e^{\bi S^{\mathbb V^n}},$ we have that 
\begin{align*}
&\int_0^T\E[1_{\Omega_c^n\cap \Omega_c}|u^{\mathbb V^n}-u^{\mathbb V}|^2]ds\\
&\le \int_0^T\sum_{i=1}^{N} 2\Big(\E[1_{\Omega_c^n\cap \Omega_c}|\sqrt{\rho^{\mathbb V^n}_i}-\sqrt{\rho^{\mathbb V}}_i|^2] +\E[1_{\Omega_c^n\cap \Omega_c}|\sqrt{\rho_i^{\mathbb V}}(e^{\bi S_i^{\mathbb V^n}}-e^{\bi S_i^{\mathbb V}})|^2]\Big)ds\\
&\le C\int_0^T\sum_{i=1}^{N} \Big(\E[1_{\Omega_c^n\cap \Omega_c}|\sqrt{\rho_i^{\mathbb V^n}}-\sqrt{\rho_i^{\mathbb V}}|^2] +\E[1_{\Omega_c^n\cap \Omega_c}|S_i^{\mathbb V^n}-S_i^{\mathbb V}|^2]\Big)ds.
\end{align*}
By applying the It\^o formula before $\tau_c^n\cap \tau_c$ and H\"older's inequality, 
we obtain that 
\begin{align*}
&|\sqrt{\rho^{\mathbb V^n}(t)}-\sqrt{\rho^{\mathbb V}(t)}|^2\\
&=\int_0^t 2\sum_{i=1}^N \sum_{j\in N(i)} \Big(\frac 1{\sqrt{\rho_i^{\mathbb V^n}}}(S_i^{\mathbb V^n}-S_j^{\mathbb V^n})\theta_{ij}(\rho^{\mathbb V^n})-\frac 1{\sqrt{\rho_i^{\mathbb V}}}(S_i^{\mathbb V}-S_j^{\mathbb V})\theta_{ij}(\rho^{\mathbb V})\Big)(\sqrt{\rho_{i}^{\mathbb V^n}}-\sqrt{\rho_{i}^{\mathbb V}})ds\\
&\le \int_0^t C (1+c)\sum_{i=1}^N\Big(|S_i^{\mathbb V^n}-S_i^{\mathbb V}|\sqrt{\rho_i^{\mathbb V^n}}-\sqrt{\rho_i^{\mathbb V}}|+|\sqrt{\rho_i^{\mathbb V^n}}-\sqrt{\rho_i^{\mathbb V}}|^2\Big)ds
\end{align*}
and that
\begin{align*}
&|S^{\mathbb V^n}(t)-S^{\mathbb V}(t)|^2\\
&=\int_0^t 2\sum_{i=1}^N\sum_{j\in N(i)}(-\frac 14(S_i^{\mathbb V^n}-S_j^{\mathbb V^n})^2+\frac 14(S_i^{\mathbb V}-S_j^{\mathbb V})^2)(S_i^{\mathbb V^n}-S_i^{\mathbb V})ds\\
&+\int_0^t 2\sum_{i=1}^N(-\mathbb V_i^n+\mathbb V_i)(S_i^{\mathbb V^n}-S_i^{\mathbb V})ds\\
&+\int_0^t 2\sum_{i=1}^N\sum_{j=1^N}(-\mathbb W_{ij}\rho_{j}^{\mathbb V^n}+\mathbb W_{ij}\rho_{j}^{\mathbb V} )(S_i^{\mathbb V^n}-S_i^{\mathbb V})ds\\
&\le \int_0^t C(1+c)\Big(|S^{\mathbb V^n}-S^{\mathbb V}|^2+|\sqrt{\rho^{\mathbb V^n}}-\sqrt{\rho^{\mathbb V}}|^2+|\mathbb V^n-\mathbb V|^2\Big)ds.
\end{align*}
The Gronwall's inequality, together with the above estimates, leads to 
\begin{align*}
\E\Big[|\sqrt{\rho^{\mathbb V^n}(t)}-\sqrt{\rho^{\mathbb V}(t)}|^2+|S^{\mathbb V^n}(t)-S^{\mathbb V}(t)|^2\Big]
&\le \exp^{\int_0^tC(1+c)ds} \int_0^t\E\big[|\mathbb V^n-\mathbb V|^2\big]ds.
\end{align*}
Taking $n\to\infty$ and then $c\to \infty$, we achieve that 
\begin{align*}
&\lim_{c\to\infty}\lim_{n\to \infty}\int_0^T\E[1_{\Omega_c^n\cap \Omega_c}|u^{\mathbb V^n}-u^{\mathbb V}|^2]ds\\
&\le \lim_{c\to\infty}\lim_{n\to \infty} \int_0^T \exp^{\int_0^tC(1+c)ds} \int_0^t\E\big[|\mathbb V^n-\mathbb V|^2\big]dsdt=0.
\end{align*}
Combining the estimate on $\Omega_c^n\cap \Omega_c$ and $\Omega/(\Omega_c^n\cap \Omega_c),$ we obtain the desired result. Similarly, one could also obtain the strong convergence of $u^{\mathbb V^n}$ in the topology 
$L^2(\Omega; \mathcal C([0,T];\mathbb C^N)).$
\end{proof}

\begin{tm}\label{sop-exi}
Let $\beta\ge 0.$
For the control problem \eqref{ocp} with the constraint \eqref{snls1} or \eqref{snls2}.
there always exists an optimal control $\mathbb V^*\in \mathcal U$ which minimizes the objective functional $J.$
\end{tm}

\begin{proof}

By Lemma \ref{reg}, to get the unique existence of an optimal control,
it suffices to show the lower continuity of $F$ if $\beta>0$, which can be obtained by using Proposition \ref{converge} and the Fatou lemma. 

In the following, we show the existence of an optimal control when $\beta=0.$
Since $\gamma  \sum_{i=1}^N |u^{\mathbb V}_i(T)-f^1_i|^2$ is bounded from below and $|\mathbb V_i|\le \alpha$ in $\mathcal U$, the infimum of $F$
exists. Let $(u^{\mathbb V^n},\mathbb V^n)$ be a minimizing sequence. 
By the a priori estimate in Proposition \ref{prop-snls}, there exists a subsequence, still denoted by $\mathbb V^n$, such that $\mathbb V^n \to \mathbb V^*$ weakly in $L^2(\Omega\times [0,T];\mathbb R^N).$
By Mazur's theorem, we have a sequence of convex combinations denoted by $\widetilde {\mathbb V}^{m}:\sum_{n\ge 1}\alpha_{nm}u_{n+m}$ with $\alpha_{nm}\ge 0, \sum_{n\ge 1}\alpha_{nm}=1$ such that 
\begin{align*}
\widetilde {\mathbb V}^{m} \to {\mathbb V^*}, \; \text{strongly in} \; L^2(\Omega\times [0,T];\mathbb R^N). 
\end{align*}
Using the fact that $|\widetilde{\mathbb V}^{m}_i|\le \alpha,$ it follows that $\mathbb V^*\in \mathcal U.$ By Proposition \ref{converge}, we also have the strong convergence, $u^{\widetilde{\mathbb V}^{m}}\to u^{\mathbb V^*}$ in $L^2(\Omega; \mathcal C([0,T];\mathbb C^N)).$ Therefore, $(u^{\mathbb V^*},\mathbb V^*)$ is admissible. By making use of the convexity of $ |u^{\mathbb V}_i(T)-f^1_i|^2, i\le N$ and the Fatou lemma, we conclude that
\begin{align*}
J(u^{\mathbb V^*})&\le \lim_{m\to\infty} J(\widetilde {\mathbb V}^{m})\le \lim_{m\to \infty} \sum_{n\ge 1}\alpha_{nm}J(\widetilde {\mathbb V}^{m})
\le \inf_{\mathbb V\in \mathcal U} J(\mathbb V),
\end{align*} 
which completes the proof.
\end{proof}

From the above procedures, it can be seen that
all the results in this subsection still hold as long as 
the cost functional in \eqref{ocp} take the form of $J(\mathbb V)=\E[f(u^{\mathbb V}(T))]+\beta \E\Big[\int_0^T \sum_{i=1}^N |\mathbb V_j(t)-Z_j(t)|^2dt\Big],
$
where the function $f$ has a lower bound and is lower semi-continuous convex. We also would like to remark that the second term $ \E\Big[\int_0^T \sum_{i=1}^N |\mathbb V_j(t)-Z_j(t)|^2dt\Big]$ could be extended to more general objective functional, like $\E \Big[\int_0^T |u^{\mathbb V}(t)-Z^1(t)|^2\Big]dt$ with an $\mathcal F_t$-adapted and $L^2$-integrable process $Z^1$, whose integrator is bounded from below and convex.

\subsection{Stochastic control problem with diffusion control}

Similar to the linear potential control problem on graph,
 we can also obtain the existence of an optimal control problem with diffusion control which has not been reported even in the continuous case. Since the proof is similar to that of Theorem \ref{sop-exi}, we omit the details and only present the main result here. 
 
 Consider the constraint \eqref{snls1} with the control variable $\sigma\in \mathbb R^N.$
 The admissible control set $\widetilde {\mathcal U}$ is defined by 
\begin{align*}
\widetilde {\mathcal U}:=&\Big\{\sigma: \Omega\times [0,T] \to \mathbb R^N \;\big|\;\sigma(t) \;\;\text {is}\;\; \mathcal F_t\text{-adapted}, \sigma \in L^2([0,T]),\\
&\text{there exists}\; \alpha>0, \; \text{such that} \;  |\sigma_j| \le \alpha\; \text{a.s.} \Big\}
\end{align*}
Here the optimal control problem is to minimize the cost functional 
\begin{align}\label{ocp1}
J(\sigma)&:=\gamma \E\Big[ \sum_{i=1}^N |u^{\sigma}_i(T)-f^1_i|^2 \Big]
+\beta \E\Big[\int_0^T \sum_{i=1}^N | \sigma_i(t)-Z_i(t)|^2dt\Big], 
\end{align}
where $\gamma,\beta\ge 0$, $f^1$ is $\mathcal F_T$-adapted satisfying $\|f^1\|_{L^2(\Omega;\mathbb C^N)}<\infty$,  $Z\in \widetilde{\mathcal U}$,  $u^{\sigma}$ is the solution of \eqref{snls1} with the control $\sigma.$ 

\begin{tm}\label{sop-exi-dif}
For the control problem \eqref{ocp1} with the constraint \eqref{snls1}, 
there always exists an optimal control $\sigma^*\in \widetilde {\mathcal U}$ which minimizes the objective functional $J.$
\end{tm}

\begin{proof}
By applying Proposition \ref{prop-snls} and repeating the steps in the proof of Proposition \ref{converge}, the lower continuity of $J$ when $\beta=0$ can be established. Therefore, the existence of optimal control is ensured by the convexity of $|u^{\sigma}_j(T)-f^1_j|^2.$
When $\beta>0$, the existence of optimal control can be guaranteed by Lemma \ref{reg}.
\end{proof}

From the above procedures, it can be seen that
all the results in this subsection still hold as long as 
the cost functional in \eqref{ocp1}  takes the form of $J(\sigma)=\E[f(u^{\sigma}(T))]+\beta \E\Big[\int_0^T \sum_{i=1}^N |\sigma_j(t)-Z_j(t)|^2dt\Big],
$
where the function $f$ has a lower bound and is lower semi-continuous convex. 
Meanwhile, we can also have the existence of optimal
potential and diffusion controls at the same time according to Theorems \ref{sop-exi} and \ref{sop-exi-dif}.

\section{Optimal condition for the stochastic control on graph}
\label{sec-condition}
As it has been pointed out in \cite{MR3462390}, compared to  nonlinear Schr\"odinger equations driven by additive noise, it is more difficult to investigate multiplicative noise. Beyond that,
for the nonlinear Schr\"odinger equation on graph, the appearance of the nonlinear Laplacian $\Delta_G$ makes it more challenging to characterize the optimal condition than the continuous control problem. 

In this section we mainly consider the following control problem
\begin{align}\label{ocp2}
J(\mathbb V)&:=\gamma \E\Big[ \sum_{i=1}^N |u^{\mathbb V}_i(T)-f^1_i|^2 \Big]+\beta_1\E \Big[\int_0^T \sum_{i=1}^N |u_i^{\mathbb V}(t)-Z_i^1(t) |^2\Big]dt
\\\nonumber 
&+\beta \E\Big[\int_0^T \sum_{i=1}^N |\mathbb V_i(t)-Z_i(t)|^2dt\Big] 
\end{align}
with the constraint \eqref{snls1} to illustrate how to derive the optimal condition on graph.  Here $\gamma\ge 0,\beta_1\ge0,\beta\ge 0$, and $Z^1$ is an $\mathcal F_t$-adapted and $L^2$-integrable process. When $\beta_1=0,$ \eqref{ocp2} degenerates into \eqref{ocp}. Our approach can be also extended to a more general smooth convex functional setting.

\subsection{Gradient formula}
In section \ref{sec-control}, we have shown the existence of optimal potential and diffusion controls. Furthermore, in this part we study the necessary optimal condition near the minimizer $\mathbb V^*$ of \eqref{ocp2} which is also called the gradient formula.

\begin{prop}\label{ne-cond}
Let $(u^{\mathbb V^*},\mathbb V^*)$ be the solution and optimal control of \eqref{ocp2}.
Then for $$\sup_{t\in [0,T]}|\mathbb V^{\epsilon}(t)-\mathbb V^*(t)|\le \epsilon, \quad \mathbb V^{\epsilon}\in \mathcal U,$$ it holds that
\begin{align}\label{err0}
\E\Big[1_{\Omega_c} \sup_{t\in [0,T]}|u^{\mathbb V^*}(t)-u^{\mathbb V^{\epsilon}}|^{p}\Big]\le C(c,u(0),T,p)\epsilon^p,
\end{align}
where $p\ge2$ and $\Omega_c=\{\sup_{i\le N}\sup_{s\in [0,T]}\frac 1{\rho_i^{\mathbb V^*}}+\sup_{i\le N}\sup_{s\in [0,T]}\frac 1{\rho_i^{\mathbb V^{\epsilon}}}\le c\}.$

Furthermore, suppose there exists $c(\epsilon)\to \infty$ such that the random variable $C(\omega)$, defined by \eqref{rand} with $\mathbb V\in \mathcal U$, satisfies 
\begin{align}\label{gra-cond}
\lim_{\epsilon \to 0}\Big[C(c(\epsilon),u(0),T,2)\epsilon+\frac 1\epsilon \mathbb P(C(\omega)\ge \frac 1{c_1}(\log(c(\epsilon))+\log(c_2)))\Big]=0,
\end{align}
then for any $\mathbb V \in \mathcal U,$ the following variational inequality  holds:
\begin{align}\label{cond}
&\lim_{c(\epsilon )\to \infty}\E\Big[1_{\Omega_{c(\epsilon)}} \Re \Big\{\int_0^T\sum_{i=1}^N\Big((u_i^{\mathbb V^*}(t)-Z_i^1(t))\overline{X_i(t)} +(\mathbb V_i^*(t)-Z_i(t))({\mathbb V_i(t)-\mathbb V^*_i(t)})\Big)dt \\\nonumber
&\quad+\sum_{i=1}^N (u_i^{\mathbb V^*}(T)-f_i^1(T))\overline{X_i(T)} \Big\}\Big]\ge 0,
\end{align}
where $X$ is the solution of the
following equation
\begin{align}\label{var-eq}
&dX_i(t)=\Big\{\frac {\bi}2 \sum_{j\in N(i)}\frac {\partial (\Delta_G u)_i}{\partial u_j}\Big|_{u=u^{\mathbb V^*}} X_j-\bi \mathbb V^*_i X_i-\bi \sum_{l=1}^N\mathbb W_{il}|u^{\mathbb V^*}_l|^2X_i\\\nonumber
&\qquad\qquad-2\bi \sum_{l=1 }^N\mathbb W_{il}\Re (\bar u^{\mathbb V^*}_l X_l)u^{\mathbb V^*}_i\Big\}dt+\Big\{-\bi u_i^{\mathbb V^*}(\mathbb V_i-\mathbb V_i^*)\Big\}dt +\Big\{ -\bi \sigma_i X_i\Big\}\circ dW(t)\\\nonumber
&X(0)=0.
\end{align}

\end{prop}

\begin{proof}
Since the admission control  set $\mathcal U$ is convex, we can use a convex perturbation to illustrate the procedures. Consider 
$\mathbb V^{\epsilon}=(1-\epsilon)\mathbb V^*+\epsilon \mathbb V.$
Define two processes $\xi(t):=\frac {u^{\mathbb V^{\epsilon}}-u^{\mathbb V}}{\epsilon}$ and $\delta \mathbb V:=\mathbb V-\mathbb V^{\epsilon}.$
Before the stopping time $\tau^c$, according to the proof of Proposition \ref{converge}, the equation of $X_i$ is well-posed since the coefficients of \eqref{var-eq} are globally Lipschitz.
By the mean value theorem, $\xi$ will satisfy 
\begin{align*}
&d\xi_i(t)=\Big\{\frac {\bi}2 \sum_{j\in N(i)} \int_0^1\frac {\partial (\Delta_G u)_i}{\partial u_j}\Big|_{u=u^{\mathbb V^*}+\kappa \epsilon \xi}d\kappa \xi_j-\int_0^1 \bi ( \mathbb V^*_i+\kappa \epsilon \delta \mathbb V_i) d\kappa \xi_i\\
&\qquad\qquad-\sum_{l=1}^N\bi \mathbb W_{il} \Big(\int_0^1 |u^{\mathbb V^*}_l+\kappa \epsilon \xi_i|^2 d\kappa\Big) \xi_i\\
&\qquad\qquad-2\sum_{l=1}^N\bi \mathbb W_{il}\int_0^1\Re ((\overline{u^{\mathbb V^*}_l+\epsilon \kappa \xi_l}) \xi_l)(u^{\mathbb V^*}_i +\epsilon \kappa \xi_i)d\kappa \Big\} dt\\\nonumber
&\qquad\qquad +\int_{0}^1\Big\{-\bi (u_i^{\mathbb V^*}+\epsilon \kappa \xi_i)(\mathbb V_i-\mathbb V_i^*)\Big\} d\kappa dt +\Big\{ -\bi \sigma_i \xi_i\Big\}\circ dW(t).
\end{align*}

Using the similar steps in the proof of Proposition \ref{converge}, on $\Omega_c$, it holds that for any $p\ge 2,$
\begin{align*}
\E\Big[1_{\Omega_c}\sup_{t\in [0,T]}|\xi(t)|^p\Big]+\E\Big[1_{\Omega_c}\sup_{t\in [0,T]}|X(t)|^p\Big]
&\le C(c,u(0),T) \E[(\int_0^T |\delta \mathbb V|^2ds)^{\frac p2}]
\end{align*} 
and that for $p\ge2,$
\begin{align*}
\E\Big[1_{\Omega_c}\sup_{t\in [0,T]}|\xi(t)-X(t)|^p\Big]&\le C(c,u(0),T,p).
\end{align*}
Thus, \eqref{err0} follows.
Here $C(c,u(0),T,p)$ is increasing with respect to $c$ satisfying $\lim_{c\to \infty} C(c,u(0),T,p)=+\infty.$ 

For convenience, let us denote 
\begin{align*}
J_{\Omega_c}(\mathbb V)
&:=\gamma \E\Big[ 1_{\Omega_c}\sum_{i=1}^N |u^{\mathbb V}_i(T)-f^1_i|^2 \Big]+\beta_1\E \Big[1_{\Omega_c}\int_0^T \sum_{i=1}^N|u^{\mathbb V}(t)-Z^1(t) |^2\Big]dt\\
&\quad
 +\beta \E\Big[1_{\Omega_c} \int_0^T \sum_{i=1}^N |\mathbb V_i(t)-Z_i(t)|^2dt\Big].
 \end{align*}
Due to the fact that $J(\mathbb V^*)\le J(\mathbb V^\epsilon),$ we obtain that 
\begin{align*}
&0\le  {J(\mathbb V^{\epsilon})-J(\mathbb V^*)}\\
&={J_{\Omega_c}(\mathbb V^{\epsilon})-J_{\Omega_c}(\mathbb V^*)}+ {J_{\Omega/\Omega_c}(\mathbb V^{\epsilon})-J_{\Omega/\Omega_c}(\mathbb V^*)}.
\end{align*}
Using the tail estimate of $1_{\Omega/\Omega_c}$ by the arguments in the proof of Proposition \ref{converge}, we get
\begin{align*}
\lim_{c\to \infty}\lim_{\epsilon \to 0}{J_{\Omega/\Omega_c}(\mathbb V^{\epsilon})-J_{\Omega/\Omega_c}(\mathbb V^*)}=0.
\end{align*}
To derive a necessary optimal condition, we need consider the speed of the convergence for $c$ and $\epsilon.$  By  \eqref{tail-est}, we have that 
\begin{align*}
J_{\Omega/\Omega_c}(\mathbb V^{\epsilon})
&\le C\mathbb P(C(\omega)\ge \frac 1{c_1}(\log(c)+\log(c_2))).
\end{align*}
By the Taylor expansion and \eqref{err0}, we have
\begin{align*}
0&\le \frac {1}{\epsilon}\Big[{J_{\Omega_c}(\mathbb V^{\epsilon})-J_{\Omega_c}(\mathbb V^*)}\Big]+ \frac 1{\epsilon}\Big[{J_{\Omega/\Omega_c}(\mathbb V^{\epsilon})-J_{\Omega/\Omega_c}(\mathbb V^*)}\Big]\\
&\le \E\Big[1_{\Omega_c} \Re \Big\{\int_0^T\sum_{i=1 }^N\Big((u_i^{\mathbb V^*}(t)-Z_i^1(t))\overline{X_i(t)} +(\mathbb V_i^*(t)-Z_i(t))({\mathbb V_i(t)-\mathbb V^*_i(t)})\Big)dt\\
&\quad+\sum_{i=1}^N (u_i^{\mathbb V^*}(T)-f_i^1(T))\overline{X_i(T)} \Big\}\Big]\\
&+ C(c,u(0),T,2)\epsilon+\frac 1\epsilon C\mathbb P(C(\omega)\ge \frac 1{c_1}(\log(c)+\log(c_2))).
\end{align*}
Using the condition \eqref{gra-cond}, there exists $c(\epsilon)\to \infty$ such that
\begin{align*}
&\lim_{c\to \infty}\E\Big[1_{\Omega_c} \Re \Big\{\int_0^T\sum_{i=1 }^N\Big((u_i^{\mathbb V^*}(t)-Z_i^1(t))\overline{X_i(t)} +(\mathbb V_i^*(t)-Z_i(t))({\mathbb V_i(t)-\mathbb V^*_i(t)})\Big)dt \\\nonumber
&\quad+\sum_{i=1}^N (u_i^{\mathbb V^*}(T)-f_i^1(T))\overline{X_i(T)} \Big\}\Big]\ge 0,
\end{align*}
which implies \eqref{cond}. 
\end{proof}

\begin{rk} If $\mathbb V^*$ is in the interior of $\mathcal U$, then \eqref{cond} becomes the equality. In general, the limit with respect to $c$ in \eqref{cond} does not commute with the expectation since the variational equation \eqref{cond} may not have a global estimate in the expectation sense and the coefficient is singular near boundary of $\mathcal{P}(G)$.
\end{rk}

Our approach is also applicable for the cost functional 
\begin{align}\label{gen-1}
J(\mathbb V)=\mathbb E [\int_0^Tg(u^{\mathbb V}(t),\mathbb V(t))dt+h(u^{\mathbb V}(T))],
\end{align}
where $g$ and $h$ are continuous convex and differentiable with bounded first derivatives.
Here we only present the result since the proof is similar to that of Proposition \ref{ne-cond}.

\begin{prop}\label{prop-gen1}
Assume that $g$ and $h$ are continuous differentiable with bounded first derivatives. Under the condition of Proposition \ref{ne-cond} with the cost functional \eqref{gen-1},
it holds that
\begin{align*}
&\lim_{c(\epsilon )\to \infty}\E\Big[1_{\Omega_{c(\epsilon)}}\Re  \Big\{\int_0^T\sum_{i=1}^N
g_{x_i}(u^{\mathbb V^*},\mathbb V^*)\overline{X_i(t)} +g_{y_i}(u^{\mathbb V^*},\mathbb V^*)({\mathbb V_i(t)-\mathbb V^*_i(t)})\Big)dt \\\nonumber
&\quad+\sum_{i=1}^N h_x(u^{\mathbb V^*})\overline{X_i(T)} \Big\}\Big]\ge 0,
\end{align*}
where $X_i$ is the solution of \eqref{var-eq}.
\end{prop}

Similarly, we could consider the diffusion control problem,  
\begin{align}\label{ocp3}
J(\sigma)&=\mathbb E [\int_0^Tg(u^{\sigma}(t),\sigma (t))dt+h(u^{\sigma }(T))].
\end{align}
with the constraint \eqref{snls1} and $\sigma\in \widetilde{\mathcal U}$. We state its optimal condition as follows and omit the detailed proof.

\begin{prop}\label{ne-cond1}
Let $(u^{\sigma^*},\sigma^*)$ be an optimal control of \eqref{ocp3}.
Then for $|\sigma^{\epsilon}-\sigma^*|\le \epsilon, \sigma^{\epsilon}\in \mathcal U$, it holds that for $p\ge2,$
\begin{align}\label{err1}
\E\Big[1_{\Omega_c} \sup_{t\in [0,T]}|u^{\sigma^*}(t)-u^{\sigma^{\epsilon}}|^2\Big]\le C(c,u(0),T,p)\epsilon^p,
\end{align}
where $\Omega_c=\{\sup_{i\le N}\sup_{s\in [0,T]}\frac 1{\rho_i^{\sigma^*}}+\sup_{i\le N}\sup_{s\in [0,T]}\frac 1{\rho_i^{\sigma^{\epsilon}}}\le c\}.$

Furthermore, suppose that the random variable $C(\omega)$, defined by \eqref{rand} with $\sigma \in \mathcal U$, satisfies that there exists $c(\epsilon)\to \infty$ such that
\begin{align}\label{gra-cond1}
\lim_{\epsilon \to 0}\Big[C(c(\epsilon),u(0),T,2)\epsilon+\frac 1\epsilon \mathbb P(C(\omega)\ge \frac 1{c_1}(\log(c(\epsilon))+\log(c_2)))\Big]=0.
\end{align}

Suppose that $(u^{\sigma^*},\sigma^*)$ is an optimal control of \eqref{ocp3}. Then for any $\sigma \in \mathcal U,$ the following variational inequality holds:
\begin{align}\label{cond1}
&\lim_{c\to \infty}\E\Big[1_{\Omega_c} \Re  \Big\{\int_0^T\sum_{i=1}^N\Big(g_{x_i}(u^{\sigma^*},\sigma)\overline{X_i(t)} +g_{y_i}(u^{\sigma^*},\sigma^*)({\sigma_i(t)-\sigma^*_i(t)})\Big)dt \\\nonumber
&\quad+\sum_{i=1}^N h_{x_i}(u_i^{\sigma^*}(T))\overline{X_i(T)} \Big\}\Big]\ge 0,
\end{align}
where $X$ is the solution of the following equation 
\begin{align}\label{var-eq1}
&dX_i(t)=\Big\{\frac {\bi}2 \sum_{j\in N(i)}\frac {\partial (\Delta_G u)_i}{\partial u_j}\Big|_{u=u^{\sigma^*}} X_j-\bi \mathbb V_i X_i-\bi \sum_{l=1}^N\mathbb W_{il}|u^{\sigma^*}_l|^2X_i-2\bi \sum_{l=1 }^N\mathbb W_{il}\Re (\bar u^{\sigma^*}_l X_l)u^{\sigma^*}_i\Big\}dt\\\nonumber
&\qquad\qquad +\Big\{ -\bi \sigma_i^* X_i-\bi u^{\sigma^*}_i(\sigma_i-\sigma^*_i) \Big\}\circ dW(t)\\\nonumber
&X(0)=0.
\end{align}

\end{prop}

The gradient formula characterizes the necessary optimal condition of the potential and diffusion control problems. However, such condition is not very useful in practice, because the variational solution depends on the control variable $\mathbb V$ or $\sigma.$

\subsection{Backward SDE}

In this subsection, we aim to give a more in-depth description on the optimal condition via the forward and backward stochastic differential equations. 
To better illustrate the procedure while clearly explaining the main idea, we use the control problem \eqref{ocp2} with $\gamma=\beta=\beta_1=1$ as an example. To this end, we need {\it a priori} estimate of the variational solution $X$ of \eqref{var-eq} such that the limit with respect to $c$ commutes with the expectation in \eqref{cond}.

\begin{prop}\label{uni-est}
Let $\sigma$ be a constant potential, i.e., $\sigma_i=\sigma_j$, and $\rho(0)\in \mathcal P_{o}(G)$, $S(0)\in \mathbb R^N$. Assume that $\mathbb V\in \mathcal U.$
Then it holds that for $p\ge 2,$
\begin{align}\label{pri-est}
&\E\Big[\sup_{t\in [0,T]}\|u^{\mathbb V}(t)\|^p\Big]\le C(u(0),T,p,\alpha),\\\nonumber
&\E\Big[\sup_{t\in [0,T]}\|X(t)\|^p\Big]\le C(u(0),T,p,\alpha).
\end{align}
\end{prop}

\begin{proof}
According to Proposition \ref{prop-snls} and the proof of Proposition \ref{converge}, it suffices to prove a uniform lower bound estimate of the density function $\rho^{\mathbb V}(t)=|u^{\mathbb V}(t)|^2.$  Since $\sigma_i=\sigma_j$, we denote $\sigma_i=\widetilde \sigma.$ Introducing $\widetilde S_i=S_i+\widetilde \sigma W(t)$,  \eqref{snls1} can be rewritten as 
\begin{align*}
& d\rho_i=\sum_{j\in N(i)}\omega_{ij}(\widetilde S_i-\widetilde S_j)\theta_{ij}(\rho) dt;\\\nonumber
& d \widetilde S_i+(\frac 12\sum_{j\in N(i)}\omega_{ij}(\widetilde S_i-\widetilde S_j)^2 \frac {\partial \theta_{ij}}{\partial \rho_i}+\frac 18 \frac {\partial}{\partial \rho_i} I(\rho)+\mathbb V_j
+\sum_{j\in N(i)}\mathbb W_{ij}\rho_j)dt=0,
\end{align*} 
which is a nonlinear Schr\"odinger equation with random inputs. Thus it follows that 
\begin{align*}
\mathcal H(\rho(t),\widetilde S(t))&:=\frac 12 \<\nabla_G \widetilde S,\nabla_G \widetilde S\>_{\theta(\rho(t))}+\mathcal V(\rho(t))+\mathcal W(\rho(t))+\frac 18 I(\rho(t))\\
&=\mathcal H(\rho(0),S(0))<\infty, \; \text{a.s.}
\end{align*}
The property of Fisher information yields that there exists a constant $c_{low}>0$ such that 
\begin{align*}
\inf_{t\ge 0}\min_{i\le N} \rho_i(t)\ge c_{low}>0, \text{a.s.}
\end{align*}
Therefore we have $\Omega_{\frac 1{c_{low}}}=\Omega$ and 
\begin{align*}
&\E\Big[\sup_{t\in [0,T]}\|u^{\mathbb V}(t)\|^p\Big]\le C(u(0),T,p,\alpha,c_{low}).
\end{align*}
The lower bound of the density function also implies that  the coefficient of \eqref{var-eq} are bounded and Lipschitz. By repeating similar steps in the proof of Proposition \ref{ne-cond}, we complete the proof.

\end{proof}

Thanks to the lower bound estimate of the density function, we are also able to derive the corresponding backward stochastic differential equation, which is also called the adjoint equation of \eqref{var-eq}.

\begin{cor}\label{well-bsde}
Let the condition of Proposition \ref{ne-cond} hold. Let $(u^{\mathbb V^*},\mathbb V^*)$ be an optimal control of \eqref{ocp2}. Then there exists an adapted solution $(Y,\mathbb Z)$ of the following system,
\begin{align}\label{bsde}
&dY_i(t)=-\Big\{\frac {\bi}2 \sum_{i\in N(j)}\overline{\frac {\partial (\Delta_G u)_j}{\partial u_i}}\Big|_{u=u^{\mathbb V^*}} Y_j-\bi \mathbb V^*_i Y_i-\bi \sum_{l=1}^N\mathbb W_{il}|u^{\mathbb V^*}_l|^2Y_i-2\sum_{l=1}^N\mathbb W_{il}\Re (\bi  u^{\mathbb V^*}_l  \bar Y_l)u^{\mathbb V^*}_i\Big\}dt\\\nonumber
&\qquad\qquad +\frac 12 \sigma_i^2 Y_i(t) dt+ \bi \sigma_i \mathbb Z_i dt + 2(u_i^{\mathbb V^*}-Z_i^1)dt +\mathbb Z_i(t) dW(t),\\\nonumber
&Y(T)=-2u^{\mathbb V^*}(T)+2f_1(T).
\end{align} 
\end{cor}

\begin{proof}
Thanks to Proposition \ref{uni-est}, the coefficients of \eqref{bsde} are Lipschitz and bounded.
Then the standard arguments in \cite[section 3]{Yong20} yield the well-posedness of the linear BSDE  \eqref{bsde}, that is, there exists a unique adapted solution $(Y,\mathbb Z).$ 
\end{proof}

Based on the above results, we are ready to characterize the optimal condition by a coupled forward--backward SDE system. 

\begin{tm}\label{fbsde-op}
Let the condition of Proposition \ref{uni-est} hold. Then the optimal control pair $(u^{\mathbb V^*},\mathbb V^*)$ satisfies the generalized stochastic Hamiltonian system consisting of \eqref{snls1}, \eqref{bsde}
with $u(0)=\sqrt{\rho(0)}e^{\bi S(0)}$, $Y(T)=-2u^{\mathbb V^*}(T)+2f_1(T)$
and the stationary condition, i.e., for arbitrary $\mathbb V$, 
\begin{align*}
\Re\<-\bi u^{\mathbb V^*} Y+2(\mathbb V^*-Z),\mathbb V-\mathbb V^*\>\ge 0, \; \text{a.e.} \; t\in [0,T],\;\; \text{a.s.}
\end{align*}
\end{tm}

\begin{proof}
For convenience, let us denote $\Re \<X,Y\>:=\Re (\sum_{i=1}^N \bar X_iY_i)$ and $\Re(a,b)=\Re(\bar a b).$
Applying It\^o's formula, we obtain that 
\begin{align*}
&d \Re \<X(t),Y(t)\>\\
&=\sum_{i=1}^N\Big\{\Re \Big(\frac {\bi}2 \sum_{j\in N(i)}\frac {\partial (\Delta_G u)_i}{\partial u_j}\Big|_{u=u^{\mathbb V^*}} X_j, Y_i\Big)-\Re(\bi \mathbb V^*_i X_i,Y_i)-\Re(\bi \sum_{l=1 }^N \mathbb W_{il}|u^{\mathbb V^*}_l|^2X_i,Y_i)\\
&-2\Re(\sum_{l=1}^N\bi \mathbb W_{il}\Re (\bar u^{\mathbb V^*}_l X_l)u^{\mathbb V^*}_i,Y_i)\Big\}dt+\sum_{i=1}^N \Big\{\Re\Big(-\frac 12 \sigma_i^2 X_i,Y_i\Big)+\sum_{i=1}^N \Re\Big(-\bi u_i^{\mathbb V^*}(\mathbb V_i-\mathbb V_i^*), Y_i\Big)\Big\}dt \\
&+\sum_{i=1}^n\Big\{\Re\Big(-\frac {\bi}2 \sum_{i\in N(j)} \overline{\frac {\partial (\Delta_G u)_j}{\partial u_i}\Big|_{u=u^{\mathbb V^*}}} Y_j, X_i\Big)+
\Re\Big(\bi \mathbb V^*_iY_i,X_i\Big)\\
&+\Re\Big(\bi \sum_{l=1}^N\mathbb W_{il}|u^{\mathbb V^*}_l|^2Y_i+2 \sum_{l=1}^N\mathbb W_{il}\Re (\bi  u^{\mathbb V^*}_l \bar Y_l)u^{\mathbb V^*}_i,X_i\Big)\Big\}dt\\\nonumber
& +\sum_{l=1}^N \Big\{\Re\Big(\frac 12 \sigma_i^2 Y_i ,X_i\Big)+ \Re(\bi \sigma_i \mathbb Z_i ,X_i)+ 2\Re \Big(u_i^{\mathbb V^*}-Z_i^1,X_i\Big)\Big\}dt\\
&+\sum_{l=1}^N \Big\{\Re \Big(\mathbb Z_i(t), Y_i(t)\Big)+\Re\Big(-\bi \sigma_i X_i, Z_i\Big) \Big\}dW(t)+\sum_{l=1}^N\Re\Big(-\bi \sigma_i X_i,\mathbb Z_i\Big)dt\\
&=\sum_{i=1}^N\Re\Big(-\bi u_i^{\mathbb V^*}(\mathbb V_i-\mathbb V_i^*), Y_i\Big) dt+\sum_{i=1}^N 2\Re \Big(u_i^{\mathbb V^*}-Z_i^1,X_i\Big)dt
+\sum_{l=1}^N \Big\{\Re \Big(\mathbb Z_i(t), Y_i(t)\Big)\\
&\quad+\Re\Big(-\bi \sigma_i X_i, Z_i\Big) \Big\}dW(t).
\end{align*}
Taking expectation yields that 
\begin{align*}
&-\E[2\Re\<u^{\mathbb V}(T)-f^1(T),X(T)\>]=\E[ \Re \<X(T),Y(T)\>]-\E[ \Re \<X(0),Y(0)\>]\\
&= \int_0^T\E\Big[\Re\<\bi u^{\mathbb V^*}(\mathbb V-\mathbb V^*), Y\>+2\Re\<u^{\mathbb V^*}-Z^1,X(t)\> \Big]dt.
\end{align*}
By using \eqref{cond}, Proposition \ref{uni-est}, and Corollary  \ref{well-bsde},
we obtain 
\begin{align*}
0&\le \E\Big[ \Big\{\int_0^T\sum_{i=1}^N 2\Big((u_i^{\mathbb V^*}(t)-Z_i^1(t))\overline{X_i(t)} +(\mathbb V_i^*(t)-Z_i(t))({\mathbb V_i(t)-\mathbb V^*_i(t)})\Big)dt \\\nonumber
&\quad+\sum_{i=1}^N 2(u_i^{\mathbb V^*}(T)-f_i^1(T))\overline{X_i(T)} \Big\}\Big]\\
&=\int_0^T\E\Big[ -\Re\<\bi u^{\mathbb V^*} Y,\mathbb V-\mathbb V^*\>+2\Re\< \mathbb V^*-Z,\mathbb V-\mathbb V^*\> \Big]dt.
\end{align*}
Thus for arbitrary $\mathbb V$, we conclude that 
\begin{align*}
\Re\<-\bi u^{\mathbb V^*} Y+2(\mathbb V^*-Z),\mathbb V-\mathbb V^*\>\ge 0, \; \text{a.e.} \; t\in [0,T],\;\; \text{a.s.}
\end{align*}

\end{proof}

Theorem \ref{fbsde-op} can be also viewed as the Pontryagin's maximum principle. 
Based on the above theorem, we propose the corresponding forward-backward stochastic differential equation (FBSDE) for \eqref{ocp2},
\begin{align}\label{fbsde}
&\bi \frac {d u_j}{dt}=-\frac 12 (\Delta_G u)_j+u_j\mathbb V_j+u_j\sum_{l=1}^N\mathbb W_{jl}|u_l|^2+\sigma_ju_j\circ dW_t,\\\nonumber
&dY_i(t)=-\Big\{\frac {\bi}2 \sum_{i\in N(j)}\frac {\partial (\Delta_G u)_j}{\partial u_i}\Big|_{u=u^{\mathbb V}} Y_j-\bi \mathbb V_i Y_i-\bi \sum_{l=1}^N\mathbb W_{il}|u^{\mathbb V}_l|^2Y_i-2\bi \sum_{l=1}^N\mathbb W_{il}\Re (\bar u^{\mathbb V}_l Y_l)u^{\mathbb V}_i\Big\}dt\\\nonumber
&\qquad\qquad +\frac 12 \sigma_i^2 Y_i(t) dt+ \bi \sigma_i \mathbb Z_i dt + 2(u_i^{\mathbb V}-Z_i^1)dt +\mathbb Z_i(t) dW(t),\\\nonumber
&u(0)=\sqrt{\rho(0)}e^{\bi S(0)}, \; Y(T)=-2u^{\mathbb V^*}(T)+2f_1(T),\; \Re\<\bi u^{\mathbb V^*} Y+2(\mathbb V^*-Z),\mathbb V-\mathbb V*\>=0.\nonumber
\end{align}
If the control problem \eqref{ocp2} admits a unique optimal control, and the stochastic generalized FBSDE also admits a unique adapted solution $(u,Y,\mathbb Z),$ then $u$ is the optimal state process and the corresponding control $\mathbb V$ is optimal.  

We also present the Pontryagin's maximum principle for \eqref{ocp3} with the constraint \eqref{snls1} and the diffusion control $\sigma_i=\sigma_j$, $i,j\le N$.

\begin{tm}\label{fbsde-op1}
Let the condition of Proposition \ref{uni-est} hold. Then the optimal control pair $(u^{\sigma^*},\sigma^*)$ satisfies the generalized stochastic Hamiltonian system consisting of \eqref{snls1}, 
and 
\begin{align}\label{bsde1}
&dY_i(t)=-\Big\{\frac {\bi}2 \sum_{i\in N(j)}\overline{\frac {\partial (\Delta_G u)_j}{\partial u_i}\Big|_{u=u^{\sigma^*}}} Y_j-\bi \mathbb V_i Y_i-\bi \sum_{l=1}^N\mathbb W_{il}|u^{\sigma^*}_l|^2Y_i-2 \sum_{l=1 }^N\mathbb W_{il}\Re (\bi u^{\sigma^*}_l \bar  Y_l)u^{\sigma^*}_i\Big\}dt\\\nonumber
&\qquad\qquad +\frac 12 \sigma_i^2 Y_i(t) dt+ \bi \sigma_i \mathbb Z_i dt + 2(u_i^{\sigma^*}-Z_i^1)dt +\mathbb Z_i(t) dW(t),\\\nonumber
&Y(T)=-2u^{\sigma^*}(T)+2f_1(T).
\end{align}
with $u(0)=\sqrt{\rho(0)}e^{\bi S(0)}$, $Y(T)=-2u^{\sigma^*}(T)+2f_1(T)$
and the stationary condition 
\begin{align*}
\Re\<-\sigma u^{\sigma}  Y-\bi u^{\sigma} Z+2(\sigma-Z),\sigma-\sigma*\>\ge 0 \; \text{a.e.} \; t\in [0,T],\;\; \text{a.s.}
\end{align*}
\end{tm}

\begin{proof}
The proof is similar to that of Theorem \ref{fbsde-op}. By applying Propositions \ref{uni-est} and \ref{ne-cond1}, we can apply It\^o formula to 
$\Re\<X(t),Y(t)\>$. Using the similar steps in the proof of  Theorem \ref{fbsde-op} and utilizing  \eqref{cond1}, we can get that 
\begin{align*}
\int_0^T\E\Big[ \Re \<-\sigma^*u^{\sigma^*} Y -\bi u^{\sigma*} Z^*,\sigma-\sigma^*\>+2\Re\< \mathbb \sigma^*-Z,\sigma-\sigma^*\> \Big]dt\ge 0,
\end{align*} 
which completes the proof.
\end{proof}

In general, if the cost functional is \eqref{gen-1} or \eqref{ocp3}, analogous analysis leads to the following results.
\begin{tm}
Let the condition of Proposition \ref{uni-est} hold. There exists an adapted solution $(Y_i,\mathbb Z)$ of
\begin{align*}
&dY_i(t)=-\Big\{\frac {\bi}2 \sum_{i\in N(j)}\overline{\frac {\partial (\Delta_G u)_j}{\partial u_i}}\Big|_{u=u^{\mathbb V^*}} Y_j-\bi \mathbb V^*_i Y_i-\bi \sum_{l=1 }^N\mathbb W_{il}|u^{\mathbb V^*}_l|^2Y_i-2\sum_{l=1}^N\mathbb W_{il}\Re (\bi  u^{\mathbb V^*}_l  \bar Y_l)u^{\mathbb V^*}_i\Big\}dt\\\nonumber
&\qquad\qquad +\frac 12 \sigma_i^2 Y_i(t) dt+ \bi \sigma_i \mathbb Z_i dt + g_x(u^{\mathbb V^*},\mathbb V^*)dt +\mathbb Z_i(t) dW(t),\\\nonumber
&Y(T)=-h_x(u^{\mathbb V^*}(T))
\end{align*}
which correspond to the stochastic control problems with the cost \eqref{gen-1} 
such that the stationary condition 
\begin{align*}
\Re\<-\bi u^{\mathbb V^*} Y+g_{y}(u^{\mathbb V^*},\mathbb V^*),\mathbb V-\mathbb V^*\>\ge 0, \; \text{a.e.} \; t\in [0,T],\;\; \text{a.s.}
\end{align*}
hold.
For the stochastic control problem of \eqref{ocp3}, there exists an adapted solution   $(Y,\mathbb Z)$ of 
\begin{align*}
&dY_i(t)=-\Big\{\frac {\bi}2 \sum_{i\in N(j)}\overline{\frac {\partial (\Delta_G u)_j}{\partial u_i}\Big|_{u=u^{\sigma^*}}} Y_j-\bi \mathbb V_i Y_i-\bi \sum_{l=1}^N\mathbb W_{il}|u^{\sigma^*}_l|^2Y_i-2 \sum_{l=1}^N\mathbb W_{il}\Re (\bi u^{\sigma^*}_l \bar  Y_l)u^{\sigma^*}_i\Big\}dt\\\nonumber
&\qquad\qquad +\frac 12 \sigma_i^2 Y_i(t) dt+ \bi \sigma_i \mathbb Z_i dt + g_x(u^{\sigma^*},\sigma^*)dt +\mathbb Z_i(t) dW(t),\\\nonumber
&Y(T)=-h_x(u^{\sigma^*}(T)).
\end{align*}
such that the stationary condition
 \begin{align*}
\Re\<-\sigma u^{\sigma}  Y-\bi u^{\sigma} \mathbb Z+g_{y}(u^{\sigma},\sigma),\sigma-\sigma^*\>\ge 0\; \text{a.e.} \; t\in [0,T],\;\; \text{a.s.}
\end{align*}
hold.
\end{tm}

It can be seen that if the $V^*$ (or $\sigma^*$) is achieved in the interior of $\mathcal U$ (or $\widetilde{\mathcal U}$), then the stationary condition could be simplified to an equality.

\section{Conclusion}
In this paper, we propose the stochastic control problem subject to stochastic nonlinear Schr\"odinger equation on graph with either a linear potential or diffusion control. From the numerical viewpoint, we demonstrate that the particular features such as the stochastic dispersion relationship, mass conservation law, moment bounds of energy of stochastic nonlinear Schr\"odinger on graph. Furthermore, we provide the gradient formula and the Pontryagin's maximum principle for stochastic nonlinear Schr\"odinger equation on graph driven by multiplicative noise. 
These may serve as a foundation of the  numerical computation for stochastic control of stochastic nonlinear Schr\"odinger equation in a continuous domain as well (see, e.g., \cite{CDZ21}). 

There are many interesting questions that remain to be tackled. 
For instance, it will be more difficult to investiagete the nonlinear potential and diffusion controls of the stochastic nonlinear Schr\"odinger equation driven by general multiplicative noise. 
Given the solutions of the FBSDEs, can this stationary condition uniquely determine the optimal control for stochastic nonlinear Schr\"odinger equation on graph? 
The stochastic control problem over density manifold, such as the mean-field game involved with the Fisher information or non-monotone coefficient, is challenging. Besides, the numerical computation has not been
addressed in the current work. We plan to explore these issues in the future work.

\bibliographystyle{plain}
\bibliography{bib}

\end{document}